\documentclass{amsart}
\usepackage[latin1]{inputenc}
\usepackage{amssymb}
\usepackage{latexsym}
\usepackage{amscd}
\usepackage{longtable}
\usepackage{amsmath}
\usepackage{mathrsfs}
\usepackage{array}
\usepackage[all]{xy}
\usepackage{a4}

\newtheorem{defn0}{Definition}[section]
\newtheorem{prop0}[defn0]{Proposition}
\newtheorem{thm0}[defn0]{Theorem}
\newtheorem{lemma0}[defn0]{Lemma}
\newtheorem{claim0}[defn0]{Claim}
\newtheorem{corollary0}[defn0]{Corollary}
\newtheorem{example0}[defn0]{Example}
\newtheorem{remark0}[defn0]{Remark}
\newtheorem{assumption0}[defn0]{Assumption}
\newtheorem{conjecture0}[defn0]{Conjecture}
\newtheorem{notation0}[defn0]{Notation}
\newtheorem{question0}[defn0]{Question}

\newenvironment{theorem}{\begin{thm0}}{\end{thm0}}
\newenvironment{lemma}{\begin{lemma0}}{\end{lemma0}}

\newenvironment{remark}{\begin{remark0}\rm}{\end{remark0}}

\newcommand{\M}{\mathrm{M}}

\newcommand{\sign}{\mathrm{sign}}

\newcommand{\GL}{{\mathrm{GL}}}
\newcommand{\SO}{{\mathrm{SO}}}
\newcommand{\End}{{\mathrm{End}}}

\newcommand{\SL}{{\mathrm {SL}}}

\newcommand{\Z}{{\mathbb Z}}
\newcommand{\A}{{\mathbb A}}

\newcommand{\Q}{{\mathbb Q}}
\newcommand{\C}{{\mathbb C}}
\newcommand{\R}{{\mathbb R}}

\newcommand{\N}{{\mathbb N}}

\newcommand{\cA}{{\mathcal A}}

\newcommand{\cH}{{\mathcal H}}

\newcommand{\cG}{{\mathcal G}}
\newcommand{\cGL}{{\mathcal GL}}

\newcommand{\cU}{{\mathcal U}}

\newcommand{\cP}{{\mathcal P}}

\newcommand{\Sym}{{\mathrm {Sym}}}

\newcommand{\ra}{{\rightarrow}}

\newcommand{\Hom}{{\mathrm {Hom}}}

\title{Eichler-Shimura isomorphism and group cohomology on arithmetic groups}
\author{Santiago Molina}

\begin{document}

\address{CRM,
Centre de Recerca Matemàtica, Barcelona, Spain}
\email{smolina@crm.cat}
\footnote{The author is supported in part by DGICYT Grant MTM2015-63829-P.}

\maketitle

\section*{Introduction}

The Eichler-Shimura isomorphism establishes a bijection between the space of modular forms and certain cohomology groups with coefficients in a space of polynomials.
More precisely, let $k\geq 2$ be an integer and let $\Gamma\subseteq\SL_2(\Z)$ be a congruence subgroup,
then we have the following isomorphism of Hecke modules
\begin{equation}\label{E-S}
M_k(\Gamma,\C)\oplus \overline{S_k(\Gamma,\C)}\simeq H^1(\Gamma,V(k)^\vee),
\end{equation}
where $V(k)^\vee$ is the dual of the $\C$-vector space of homogenous polynomials of degree $k-2$, $M_k(\Gamma,\C)$ is the space of modular forms of weight $k$ and $S_k(\Gamma,\C)\subset M_k(\Gamma,\C)$ is the subspace of cuspidal modular forms (see \cite[Thm. 8.4]{Shim} and \cite[Thm. 6.3.4]{Hida}).

This isomorphism can be interpreted in geometric terms. Indeed, a modular form of weight $k$ can be interpreted as a section of certain sheaf of differential forms on the open modular curve attached to $\Gamma$. With this in mind, the Eichler-Shimura isomorphism can be obtained comparing deRham and singular cohomology, noticing that the singular cohomology of the open modular curve is given by the group cohomology $H^\bullet(\Gamma,V(k)^\vee)$. The aim of this paper is to omit this geometric interpretation and to provide a new group cohomological interpretation.

The restriction of the Eichler-Shimura isomorphism to the spaces of cuspidal modular forms is given by the morphisms
\[
\partial^\pm:S_k(\Gamma,\C)\longrightarrow H^1(\Gamma,V(k)^\vee),
\]
where
\[
\partial^\pm(f)(\gamma)(P)=\int_{z_0}^{\gamma z_0}P(1,-\tau)f(\tau)d\tau\pm \int_{z_0}^{\gamma z_0}P(1,\bar\tau)f(-\bar\tau)d(-\bar\tau).
\]
for any $z_0$ in $\cH$ the Poincaré hyperplane, $\gamma\in\Gamma$ and $P\in V(k)$. In fact, the morphism defining the cuspidal part of \eqref{E-S} is given by
\[
\begin{array}{ccc}
S_k(\Gamma,\C)\oplus \overline{S_k(\Gamma,\C)}&\longrightarrow &H^1(\Gamma,V(k)^\vee)\\
 (f_1,\bar f_2)&\longmapsto& (\partial^++\partial^-)(f_1)+(\partial^+-\partial^-)(f_2^*),
\end{array}
\]
where $f_2^*\in S_k(\Gamma,\C)$ is obtained from $f_2$ by complex conjugating its Fourier coefficients.
The image of $\partial^\pm$ lies in the subspaces $H^1(\Gamma,V(k)^\vee)^\pm\subset H^1(\Gamma,V(k)^\vee)$ where the natural action of $\omega=\left(\begin{array}{cc}1&0\\0&-1\end{array}\right)$ normalizing $\Gamma$ acts by $\pm 1$.

In a general setting, $F$ is a totally real number field of degree $d$, $G$ is the multiplicative group of a quaternion algebra $A$ over $F$, and $\phi$ is a weight $\underline{k}=(k_1,\cdots, k_d)$ cuspidal automorphic form of $G$ with level $\cU\subset G(\A^\infty)$ and central character $\psi:\A^\times/F^\times\rightarrow\C^\times$. We assume that $\psi_{\sigma_i}(x)={\rm sign}(x)^{k_i}|x|^{\mu_i}$, for any archimedean place $\sigma_i:F\hookrightarrow\R$.
In this scenario, by an automorphic form we mean a function on $\cH^r\times G(\A^\infty)/\cU$, where $\cH$ is the Poincaré upperplane and $r$ is the cardinal of the set $\Sigma$ of archimedean places where $A$ splits, with values in $\bigotimes_{\sigma_i\in\infty\setminus\Sigma}V_{\mu_i}(k_i)^\vee$ (see \S \ref{secdiscser} for a precise definition of $V_{\mu_i}(k_i)$), that satisfies the usual transformation laws with respect to the weight-$k_i$-actions of $G(F)$. 
The interesting cohomology subgroups to consider are:
\[
H^r(G(F)^+,\cA(V_\psi(\underline{k}),\C)^\cU)=\bigoplus_{i=1}^n H^r(\Gamma_{g_i},V_\psi(\underline{k})^\vee),\qquad \Gamma_{g_i}=G(F)^+\cap g_i\cU g_i^{-1},
\]
where $V_\psi(\underline{k})$ is the tensor product of the polynomial spaces $V_{\mu_i}(k_j)$ ($j=1,\cdots,d$), $G(F)^+\subseteq G(F)$ is the subgroup of totally positive elements, $\{g_i\}_{i=1,\cdots,n}\subset G(\A^\infty)$ is a set of representatives of the double coset space $G(F)^+\backslash G(\A^\infty)/\cU$,  and $\cA(V_\psi(\underline{k}),\C)^\cU=C(G(\A^\infty)/\cU,V_\psi(\underline{k})^\vee)$. Similarly as in the classical case, for any character $\varepsilon:G(F)/G(F)^+\longrightarrow\pm1$ we can define a morphism
\[
\partial^\varepsilon:S_{\underline{k}}(\cU,\psi)\longrightarrow H^r(G(F)^+,\cA(V_\psi(\underline{k}),\C)^\cU)({\varepsilon}),
\]
from the set $S_{\underline{k}}(\cU,\psi)$ of automorphic cuspforms of weight $\underline{k}$, level $\cU$ and central character $\psi$, to the $\varepsilon$-isotypical component of $H^r(G(F)^+,\cA(V_\psi(\underline{k}),\C)^\cU)$. Such a map is given by:
\[
\partial^\epsilon\phi=\sum_{\gamma\in G(F)/G(F)^+}\epsilon(\gamma)\partial\phi^\gamma,
\]
where $\partial\phi\in H^r(G(F)^+,\cA(V_\psi(\underline{k}),\C))$ is the class of the cocycle
\begin{equation*}\label{cocycleex}
(G(F)^+)^r\ni(g_1,g_2,\cdots, g_r)\longmapsto \int_{\tau_1}^{g_1\tau_1}\cdots\int^{g_1\cdots g_r\tau_r}_{g_1\cdots g_{r-1}\tau_r}P_{\Sigma}(1,-\underline{z})\langle\phi(\underline{z},g),P^{\Sigma}\rangle d\underline{z},
\end{equation*}
with $P_{\Sigma}\in\bigotimes_{j=1}^rV_{\mu_j}(k_j)$, $P^{\Sigma}\in \bigotimes_{j=r+1}^dV_{\mu_j}(k_j)$ and $\underline{z}=(z_1,\cdots,z_r),(\tau_1,\cdots,\tau_r)\in\cH^r$. Our result will provide a group cohomological interpretation to the morphisms $\partial^\varepsilon$, for any character $\varepsilon$.

Let $F_\infty\simeq\R^d$ be the product of the archimedean completions of $F$, let $\cG_\infty$ be the Lie algebra of $G(F_\infty)$ and let $K_\infty\subseteq G(F_\infty)$ be a maximal compact subgroup. 
Then the $(\cG_\infty,K_\infty)$-module generated by $\phi$ is isomorphic to $D_\psi(\underline{k})$, the tensor product of discrete series of weight $k_j$ at archimedean places in $\Sigma$ and polynomial spaces $V_{\mu_j}(k_j)$ at archimedean places not in $\Sigma$. This implies that any $\phi\in S_{\underline{k}}(\cU,\psi)$ provides an element
\[
\phi\in H^0(G(F),\cA(D_\psi(\underline{k}),\C)^\cU);\qquad \cA(D_\psi(\underline{k}),\C)^\cU:=\Hom_{(\cG_\infty,K_\infty)}(D_\psi(\underline{k}),\cA^\cU),
\]
where $\cA^\cU$ is the $(\cG_\infty,K_\infty)$-module of smooth admissible functions $f:G(\A)/\cU\rightarrow\C$. Our main result (Theorem \ref{mainres}) can be rewritten as follows: 
\begin{theorem}
There exists an exact sequence of $G(F)$-modules
\begin{equation*}\label{exseqs}
0\rightarrow\cA(V_\psi(\underline{k})(\varepsilon),\C)^\cU \rightarrow \cA(I^{\varepsilon}_1(\underline{k}),\C)^\cU\rightarrow \cA(I^{\varepsilon}_2(\underline{k}),\C)^\cU\rightarrow\cdots\rightarrow  \cA(D_\psi(\underline{k}),\C)^\cU\rightarrow 0,
\end{equation*}
such that, up to an explicit constant, the morphism $\partial^\epsilon$ is given by the corresponding connection morphism
\[
H^{0}(G(F),\cA(D_\psi(\underline{k}),\C)^\cU)\longrightarrow H^{r}(G(F),\cA(V_\psi(\underline{k})(\varepsilon),\C)^\cU)\simeq H^{r}(G(F)^+,\cA(V_\psi(\underline{k}),\C)^\cU)(\varepsilon).
\]
\end{theorem}

We obtain the above exact sequence from extensions of the $(\cG_{\sigma_i},K_{\sigma_i})$-modules of discrete series $D_{\mu_i}(k_i)$ at every place $\sigma_i\in \Sigma$. The archimedean local nature of these connection morphisms $\partial^{\varepsilon}$ implies that the $G(\A^\infty)$-representation generated by $\partial^\varepsilon\phi$ coincides with the restriction to $G(\A^\infty)$ of $\pi_\phi$, the automorphic representation attached to $\phi$.

The image of a cuspidal automorphic representation $\pi_\phi$ through the morphisms $\partial^\varepsilon$ is used in many papers to give a group cohomological construction of cyclotomic and anti-cyclotomic $p$-adic $L$-functions and Stickelberger elements attached to quadratic extensions of a totally real number field (see for instance \cite{Spiess}, \cite{Mol} and \cite{BG}). The explicit form of $\partial^\varepsilon$ given in Theorem \ref{mainres} provides the interpolation properties of these objects.

Another application is the construction of Stark-Heegner points. By means of the connection morphisms $\partial^{\varepsilon^{\pm 1}}$, where $(\varepsilon^+,\varepsilon^-)$ is a well chosen pair of characters, one can construct a
complex torus $\C^{[L:\Q]}/\Lambda$ attached to a weight $2$ automorphic representation $\pi_\phi$ with field of coefficients $L$. It is conjectured that such complex torus coincides with the abelian variety of $\GL_2$-type attached to $\pi_\phi$. In \cite{GMM}, we use the cohomological description of $\partial^{\varepsilon^{\pm 1}}$ to construct Stark-Heegner points in the complex torus, that we conjecture to be global points in the corresponding abelian variety. Such points are conjecturally defined over class fields of quadratic extensions of $F$ and satisfy explicit reciprocity laws.

\subsection*{Notation}

Throughout this paper, we will denote by $\int_{S^1}d\theta=\int_{\SO(2)}d\theta$ the Haar measure of $S^1=\SO(2)$ such that ${\rm vol}(S^1)=\pi$.

Let $F$ be a number field. For any place $v$ of $F$, we denote by $F_v$ its completion at $v$.
Given a finite set of places $S$ of $F$, we denote by $F_S$ the product of completions at every place in $S$. We denote by $F_\infty$ the product of completions at every archimedean place. Similarly, for any subset $\Sigma$ of archimedean places, $F_{\infty\setminus\Sigma}$ will be the product of completions at every archimedean place not in $\Sigma$.
We denote by $\A$ the ring of adeles of $F$. For any set $S$ of places of $F$, we write $\A^S$ for the ring adeles outside $S$. Consistent with this notation, we denote by $\A^\infty$ the ring of finite adeles of $F$.

\section{Discrete series}

\subsection{Finite dimensional representations}\label{Fin}

Let $A$ be a quaternion algebra defined over a local field $F$.
Let $K/F$ be an extension where $A$ splits. For any natural number $k\in\N$, 
let $\cP_{k-2}^K\simeq\Sym^{k-2}(K^2)$ be the finite $K$-vector space of homogeneous polynomials of degree $k-2$. We have a well defined action of $\GL_2(K)$ on $\cP_{k-2}^K$ given by
\begin{equation}\label{actpol}
\left( \left(\begin{array}{cc}a&b\\c&d\end{array}\right)\ast P\right)(x,y):=P\left((x,y)\left(\begin{array}{cc}a&b\\c&d\end{array}\right)\right)=P\left(ax+cy,bx+dy\right).
\end{equation}
If we fix an embedding
\[
\iota:A\longrightarrow A\otimes_F K\simeq\M_2(K),
\]
then $\cP_{k-2}^K$ is equipped with an action of $A^\times$.

We denote by $\det:A\ra F$ the reduced norm of $A$, and let us consider
\[
A^+=\{a\in A:\;\det(a)\in F^2\}.
\]
We write $V(k)_K=\cP_{k-2}^K\otimes {\rm det}^{\frac{2-k}{2}}$ with the natural action of $A^+$. It is clear that the centre of $A^\times$ acts trivially on $V(k)_K$. Notice that, if $k$ is even, the action of $A^+$ on $V(k)_K$ extends to a natural action of $A$.

\subsection{Discrete series and exact sequences}\label{secdiscser}

Assume that $F=\R$ and $A=\M_2(\R)$. Let $\cGL_2(\R)$ be the Lie algebra of $\GL_2(\R)$. 
For any $k\in \Z$ and $\mu\in \C$, we define $I_\mu(k)$ as the $(\cGL_2(\R),\SO(2))$-module of smooth admissible vectors in
\[
\left\{f:\GL_2(\R)^+\rightarrow\C:\;f\left(\left(\begin{array}{cc}t_1&x\\&t_2\end{array}\right)g\right)={\rm sign}(t_1)^k(t_1t_2)^{\frac{\mu}{2}}\left(\frac{t_1}{t_2}\right)^{\frac{k}{2}}f(g)\right\}.
\]

Write $V_\mu(k)=V(k)_\C\otimes\det^{\frac{\mu}{2}}$.
If we assume that $k\geq 2$, we have the well defined morphism of $(\cGL_2(\R),\SO(2))$-modules
\[
\iota:V_{\mu}(k)\longrightarrow I_{\mu}(2-k);\qquad \iota(P)\left(\begin{array}{cc}a&b\\c&d\end{array}\right)=(ad-bc)^{\frac{2-k+\mu}{2}}P(c,d).
\]
Moreover, we have a $(\cGL_2(\R),\SO(2))$-invariant pairing (see \cite[\S 2]{Bump})
\[
\langle\;,\;\rangle_I:I_\mu(k)\times I_{-\mu}(2-k)\longrightarrow\C;\qquad (f,g)\longmapsto\int_{\SO(2)}f(\theta)g(\theta)d\theta,
\]
providing the morphism of $(\cGL_2(\R),\SO(2))$-modules
\[
\bar\varphi:I_\mu(k)\longrightarrow V_{-\mu}(k)^\vee;\qquad \langle\bar\varphi(f),P\rangle_V:=\langle f,\iota(P)\rangle_I.
\]
Composing with the natural $\GL_2(\R)^+$-morphism
\begin{equation}\label{dualV}
V_{-\mu}(k)^\vee\longrightarrow V_\mu(k);\qquad F\longmapsto P_F(x,y)=\langle F,(Yx-Xy)^{k-2}\rangle_{V_{(X,Y)}},
\end{equation}
we obtain a map
\begin{eqnarray*}
\varphi:I_\mu(k)&\longrightarrow& V_\mu(k);\\
\varphi(f)(x,y)&=&\langle\bar\varphi(f),(Yx-Xy)^{k-2}\rangle_{V_{(X,Y)}}=\int_{S^1}f(\theta)(y\sin\theta +x\cos\theta)^{k-2}d\theta
\end{eqnarray*}

\begin{remark}\label{SymmProd}
Notice that we have the symmetry
\begin{eqnarray*}
\langle F,P_G\rangle_{V}&=&\langle F,\langle G,(Yx-Xy)^{k-2}\rangle_{V_{(X,Y)}}\rangle_{V_{(x,y)}}\\
&=&(-1)^k\langle G,\langle F,(Xy-Yx)^{k-2}\rangle_{V_{(x,y)}}\rangle_{V_{(X,Y)}}=(-1)^k\langle G,P_F\rangle_{V},
\end{eqnarray*}
for all $F,G\in V_\mu(k)^\vee$.
\end{remark}

The kernel of $\varphi$ is $D_\mu(k)$ the \emph{Discrete Series} $(\cGL_2(\R),O(2))$-module of weight $k$ and central character $x\mapsto {\rm sign}(x)^k|x|^\mu$. 
This definition implies that $D_\mu(k)$ lies in the following exact sequence of $(\cGL_2(\R),SO(2))$-modules:
\begin{equation}\label{exseqdisc}
0\longrightarrow D_\mu(k)\longrightarrow I_\mu(k)\stackrel{\varphi}{\longrightarrow} V_\mu(k)\longrightarrow 0.
\end{equation}

Since any $g\in\GL_2(\R)^+$ can be written uniquely as $g=u\cdot \tau(x,y)\cdot\kappa(\theta)$, where
\[
u\in\R^+,\;\tau(x,y)=\left(\begin{array}{cc}y^{1/2}&xy^{-1/2}\\&y^{-1/2}\end{array}\right)\in B,\;\kappa(\theta)=\left(\begin{array}{cc}\cos\theta&\sin\theta\\ -\sin\theta&\cos\theta\end{array}\right)\in\SO(2),
\]
we have that
\[
I_\mu(k)=\bigoplus_{t\equiv k\;({\rm mod}\;2)}\C f_{t};\qquad f_{t}(u\cdot \tau(x,y)\cdot\kappa(\theta))=u^\mu y^{\frac{k}{2}}e^{ti\theta}.
\]
The $(\cGL_2(\R),\SO(2))$-module structure of $I_s(k)$ can be described as follows: Let $L,R\in\cGL_2(\R)$ be the \emph{Maass differential operators} defined in \cite[\S 2.2]{Bump}
\[
 L=e^{-2i\theta}\left(-iy\frac{\partial}{\partial x}+y\frac{\partial}{\partial y}-\frac{1}{2i}\frac{\partial}{\partial \theta}\right),\quad R=e^{2i\theta}\left(iy\frac{\partial}{\partial x}+y\frac{\partial}{\partial y}+\frac{1}{2i}\frac{\partial}{\partial \theta}\right).
\]
Then, the $(\cGL_2(\R),\SO(2))$-module $I_s(k)$ is characterized by the relations:
\begin{eqnarray}
Rf_{t}=\left(\frac{k+t}{2}\right)f_{t+2};&\qquad& Lf_{t}=\left(\frac{k-t}{2}\right)f_{t-2}; \label{rel1}\\
\kappa(\theta)f_{t}=e^{ti\theta}f_{t};&\qquad&u f_{t}=u^\mu f_{t},\label{rel2}
\end{eqnarray}
for any $\kappa(\theta)\in\SO(2)$ and $u\in \R^{+}\subset\GL_2(\R)^+$. 

Write $z=x+iy$ and $\bar z=x-iy$. For $n\in\{0,1,\cdots, k-2\}$, let us consider the elements $P_n\in V_\mu(k)$, $P_n(x,y)=z^n\bar z^{k-2-n}$. It is clear that $\{P_n\}_{n=0,\cdots,k-2}$ is a basis for the $\C$-vector space $V_\mu(k)$. We compute that
\begin{eqnarray*}
\varphi(f)(x,y)&=&\int_{S^1}f(\theta)((2i^{-1})(z-\bar z)\sin\theta +2^{-1}(z+\bar z)\cos\theta)^{k-2}d\theta\\
&=&\int_{S^1}2^{2-k}f(\theta)(ze^{- i\theta}+\bar ze^{ i\theta})^{k-2}d\theta\\
&=&2^{2-k}\sum_{n=0}^{k-2}\binom{k-2}{n}P_n(x,y)\int_{S^1}f(\theta)e^{-i(2n-k+2)\theta}d\theta
\end{eqnarray*}
By orthogonality, we deduce that $\varphi(f_{2n-k+2})=2^{2-k}\pi\binom{k-2}{n}P_{n}(x,y)$.


Since $\kappa(\theta)P_n=e^{(2n-k+2)i\theta}P_n$, the morphism of $\C$-vector spaces
\begin{equation}\label{sect}
s:V_\mu(k)\longrightarrow I_\mu(k); \qquad s(P_n)=\frac{2^{k-2}}{\pi}\binom{k-2}{n}^{-1}f_{2n-k+2},
\end{equation}
defines a section of $\varphi$ as $\SO(2)\R^+$-modules.

\begin{remark}\label{prodV}
Since \eqref{dualV} is an isomorphism, we can define a non-degenerate $\GL_2(\R)^+$-invariant bilinear pairing $V_\mu(k)\times V_{-\mu}(k)\rightarrow\C$
\[
\langle P_F,Q\rangle=\langle F,Q\rangle_V,\qquad F\in V_{-\mu}(k)_\C^\vee,\quad Q\in V_{-\mu}(k),
\]
which is symmetric or antisymmetric depending on the parity of $k$, by Remark \ref{SymmProd}. Moreover, by the definition of \eqref{dualV},
\[
P(s,t)=\langle P,Q_{s,t}\rangle,\qquad Q_{s,t}(x,y)=(ys-xt)^{k-2}.
\]
In particular,
\begin{equation}\label{P-1tau}
P(-1, i)=i^{2-k}\langle P,P_0\rangle=i^{k-2}\langle P_{0},P\rangle.
\end{equation}
Since $s$ is a section of $\varphi$, we compute
\begin{equation}\label{prodVprodI}
\langle P,Q\rangle=\langle \bar\varphi(s(P)),Q\rangle_V=\langle s(P),\iota(Q)\rangle_I,
\end{equation}
for all $P,Q\in V_\mu(k)$. 
\end{remark}

\subsection{The $(\cGL_2(\R),O(2))$-module of discrete series}

We want to give structure of $(\cGL_2(\R),O(2))$-module to $I_\mu(k)$.
Hence, we have to define the action of $\omega=\left(\begin{array}{cc}1&\\&-1\end{array}\right)\in O(2)\setminus \SO(2)$. That is to say, we have to define
$\omega\in\End(I_\mu(k))$ such that
\[
(i)\quad\omega f_{t}\in\C f_{-t};\qquad (ii)\quad\omega^2=1;\qquad (iii)\quad\omega R=L\omega.
\]
If we write $\omega f_{t} = \lambda(t)f_{-t}$, condition (ii) implies that $\lambda(t)\lambda(-t)= 1$. Moreover, condition (iii) implies that $\lambda(t) = \lambda(t+2)$. We obtain two possible $(\cGL_2(\R),O(2))$-module structures for $I_\mu(k)$: Letting $\lambda(t)=1$ for all $t \equiv k\;{\rm mod}\;2$, or letting $\lambda(t)=-1$ for all $t \equiv k\;{\rm mod}\;2$. Write $I_\mu(k)^\pm$ for the $(\cGL_2(\R),O(2))$-module such that $\omega f_{t} =\pm f_{-t}$, respectively.

By abuse of notation, write also $V_\mu(k)$ and $V_\mu(k)_\R$ (in case $\mu\in\R$) for the $\GL_2(\R)$-representations
\[
V_\mu(k)=V_\mu(k)_\C=\cP_{k-2}^\C\otimes \mid\det\mid^{\frac{2-k+\mu}{2}},\qquad V_\mu(k)_\R=\cP_{k-2}^\R\otimes \mid\det\mid^{\frac{2-k+\mu}{2}}.
\]
\begin{remark}\label{no-G-pairing}
With this $\GL_2(\R)$-module structure, the pairing $\langle\cdot,\cdot\rangle$ introduced in Remark \ref{prodV} is not a $\GL_2(\R)$-invariant in general. In fact one can show that
\[
\langle gP,Q\rangle=\sign(\det g)^k\langle P,g^{-1}Q\rangle.
\]
\end{remark}
Note that, for any $f\in I_\mu(k)^\pm$, we have that $\omega f(\theta)=\pm f(-\theta)$,
hence we compute that,
\begin{eqnarray*}
\varphi(\omega f)(x,y)&=&\int_{S^1}\omega f(\theta)(x\cos\theta +y\sin\theta)^{k-2}d\theta\\
&=&\pm\int_{S^1}f(-\theta)(x\cos\theta +y\sin\theta)^{k-2}d\theta\\
&=&\pm\int_{S^1}f(\theta)(x\cos\theta -y\sin\theta)^{k-2}d\theta=\pm\omega(\varphi(f))(x,y)
\end{eqnarray*}
This implies that the exact sequence of $(\cGL_2(\R),SO(2))$-modules \eqref{exseqdisc} provides the exact sequences of $(\cGL_2(\R),O(2))$-modules
\begin{eqnarray}
0\longrightarrow D_\mu(k)\stackrel{\iota}{\longrightarrow} &I_\mu(k)^+&\longrightarrow V_\mu(k)\longrightarrow 0,\label{exseqdisc1}\\
0\longrightarrow D_\mu(k)\stackrel{\iota\circ I}{\longrightarrow} &I_\mu(k)^-&\longrightarrow V_\mu(k)(\varepsilon)\longrightarrow 0,\label{exseqdisc2}
\end{eqnarray}
where $\varepsilon:\GL_2(\R)\rightarrow{\pm 1}$ is the character given by $\varepsilon(g)=\sign\det(g)$, $D_\mu(k)$ is the $(\cGL_2(\R),O(2))$-module with fixed action of $\omega$ given by $\omega f(\theta) =f(\theta)$, and $I$ is the automorphism of  $(\cGL_2(\R),\SO(2))$-modules
\[
I:D_\mu(k)\longrightarrow D_\mu(k):\quad I(f_{t})=\sign(t)f_{t}.
\]
Note that $\iota\circ I$ is a monomorphism of  $(\cGL_2(\R),O(2))$-modules because $I(\omega f)=-\omega(I(f))$.

\subsection{Matrix coefficients}\label{matcoef}

Let us consider $A(\C)$ the $(\cGL_2(\R),\SO(2))$-module of admissible $C^\infty$ functions $f:\GL_2(\R)^+\rightarrow\C$. For any $f_0\in I_{-\mu}(2-k)$, I claim that
\[
\varphi_{f_0}:I_\mu(k)\longrightarrow A(\C),\qquad \varphi_{f_0}(f)(g_\infty)=\langle g_\infty f,f_0\rangle_I,\qquad g_\infty\in\GL_2(\R)^+,
\]
provides a well defines morphism of $(\cGL_2(\R),\SO(2))$-modules. Indeed, for any element of the Lie algebra $G\in \cGL_2(\R)$,
\begin{eqnarray*}
G\varphi_{f_0}(f)(g_\infty)&=&\frac{d}{d t}(\varphi_{f_0}(g_\infty\exp(tG)))\mid_{t=0}=\frac{d}{d t}(\langle g_\infty\exp(tG) f,f_0\rangle_I)\mid_{t=0}\\
&=&\varphi_{f_0}(Gf)(g_\infty).
\end{eqnarray*}

\subsection{$\R$-structures of Discrete series}

As we can see in \cite[\S 2.2]{Bump}, $R$ and $L$ are not in $\cGL_2(\R)$, they are \emph{Caley transformations} in $\cGL_2(\C)$ of elements in $\cGL_2(\R)$. In fact, $\cGL_2(\R)$ is generated by
\begin{eqnarray*}
R+L&=&-2y\sin(2\theta)\frac{\partial}{\partial x}+2y\cos(2\theta)\frac{\partial}{\partial y}+\sin(2\theta)\frac{\partial}{\partial \theta};\qquad u\frac{\partial}{\partial u};\\
 i(R-L)&=&-2y\cos(2\theta)\frac{\partial}{\partial x}-2y\sin(2\theta)\frac{\partial}{\partial y}+\cos(2\theta)\frac{\partial}{\partial \theta};\quad
 \mbox{and}\quad \frac{\partial}{\partial \theta}.
\end{eqnarray*}
If we define
$h_{t}:=f_{t}+f_{-t}\in I_\mu(k)^{\pm}$ and $g_{t}:=i(f_{t}-f_{-t})\in I_\mu(k)^{\pm}$, it is easy to compute that
\begin{eqnarray*}
(R+L)h_{t}=\left(\frac{k+t}{2}\right) h_{t+2}+\left(\frac{k-t}{2}\right)h_{t-2},&\quad& \frac{\partial}{\partial\theta}h_{t}=-t g_{t},\\
i(R-L)h_{t}=\left(\frac{k+t}{2}\right) g_{t+2}-\left(\frac{k-t}{2}\right)g_{t-2},&\quad&\omega h_{t}=\pm h_{t} \\
\kappa(\theta)h_{t}=\cos(t\theta)h_{t}-\sin(t\theta)g_{t},&\quad&\omega g_{t}=\mp g_{t}.
\end{eqnarray*}
Hence the $\R$-vector space $I_\mu(k)^\pm_\R\subset I_\mu(k)^\pm$ generated by $h_{t}$ and $g_{t}$ defines a $(\cGL_2(\R),O(2))$-module over $\R$.

We check that the morphisms $\varphi:I_\mu(k)^+\ra V_\mu(k)_\C$ and $\varphi:I_\mu(k)^-\ra V_\mu(k)(\varepsilon)_\C$ descend to morphisms of $(\cGL_2(\R),O(2))$-modules over $\R$
\[
\varphi^+:I_\mu(k)_\R^+\longrightarrow V_\mu(k)_\R,\qquad \varphi^-:I_\mu(k)_\R^-\longrightarrow V_\mu(k)_\R(\varepsilon).
\]
Hence the kernel $D_\mu(k)_\R\subset D_\mu(k)$ of $\varphi^+$ defines a $(\cGL_2(\R),O(2))$-module over $\R$, generated by $h_{k}$, such that $D_\mu(k)_\R\otimes_\R\C=D_\mu(k)$.
Nevertheless, the automorphism of  $(\cGL_2(\R),SO(2))$-modules $I:D_\mu(k)\ra D_\mu(k)$ does not descend to an automorphism of  $(\cGL_2(\R),SO(2))$-modules over $\R$ since $I(h_{t})=-\sign(t)ig_{t}$. In fact,
\[
I(D_\mu(k)_\R)=iD_\mu(k)_\R\subset D_\mu(k).
\]
We obtain the exact sequences of $(\cGL_2(\R),O(2))$-modules over $\R$
\begin{eqnarray}
0\longrightarrow D_\mu(k)_\R\stackrel{\iota}{\longrightarrow} &I_\mu(k)_\R^+&\longrightarrow V_\mu(k)_\R\longrightarrow 0,\label{exseqdiscreal1}\\
0\longrightarrow D_\mu(k)_\R\stackrel{\iota\circ I}{\longrightarrow} &i I_\mu(k)_\R^-&\longrightarrow i V_\mu(k)_\R(\varepsilon)\longrightarrow 0.\label{exseqdiscreal2}
\end{eqnarray}

\section{Connection morphisms}\label{Conn}

In this section, we assume that $G$ is the multiplicative group of a quaternion algebra that splits at the set of archimedean places $\Sigma$. Write $r=\#\Sigma$.
Let us consider the $\C$-vector space $\cA(\C)$ of functions $f:G(\A)\longrightarrow\C$ such that:
\begin{itemize}
\item There exists an open compact subgroup $U\subseteq G(\A^\infty)$ such that $f(g U)=f(g)$, for all $g\in G(\A)$.

\item Under a fixed identification $G(F_{\Sigma})\simeq\GL_2(\R)^r$, $f\mid_{G(F_{\Sigma})}\in C^\infty(\GL_2(\R)^r,\C)$.

\item Fixing $K_{\Sigma}$, a maximal compact subgroup of $G(F_{\Sigma})$ isomorphic to $O(2)^r$, we assume that any $f\in\cA(\C)$ is $K_{\Sigma}$-finite, namely, its right translates by elements of $K_{\Sigma}$ span a finite-dimensional vector space.

\item We assume that any $f\in\cA(\C)$ is ${\mathcal{Z}}$-finite, where ${\mathcal{Z}}$ is the centre of the universal enveloping algebra of $G(F_{\Sigma})$.
\end{itemize}
Write $\rho$ for the action of $G(\A)$ given by right translation, then $(\cA(\C),\rho)$ defines a smooth $G(\A^\infty)$-representation and a $(\cG_\infty,K_\infty)$-module, where $\cG_\infty$ is the Lie algebra of $G(F_{\Sigma})$ and $K_\infty=K_{\Sigma}\times G(F_{\infty\setminus{\Sigma}})$. Moreover, $\cA(\C)$ is also equipped with the $G(F)$-action:
\[
(h\cdot f)(g)=f(h^{-1}  g),\quad h\in G(F),
\]
where $g\in G(\A)$, $f\in \cA(\C)$.
Let us fix an isomorphism $G(F_{\Sigma})\simeq\GL_2(\R)^r$ that maps $K_{\Sigma}$ to $O(2)^r$ and let $V$ be a $(\cG_{\infty},K_{\infty})$-module. We define
\[
\cA(V,\C):=\Hom_{(\cG_{\infty},K_{\infty})}(V,\cA(\C)),
\]
endowed with the natural $G(F)$- and $G(\A^\infty)$-actions.

\begin{remark}
Note that if the $(\cG_\infty,K_\infty)$-module $V$ comes from a finite dimensional $G(F_\infty)$-representation $V$,
\[
\cA(V,\C)\simeq C(G(\A^\infty),\Hom(V,\C))=C(G(\A^\infty),V^\vee),
\]
where $V^\vee$ is seen as a $G(F)$-module by means of the usual injection $G(F)\hookrightarrow G(F_\infty)$, and the action of $G(F)$ on $C(G(\A^\infty),V^\vee)$ is given by $(h\ast f)(g)=h(f(h^{-1}g))$.
\end{remark}

Fix $\sigma\in \Sigma$, $\mu\in\C$ and let us consider $D_{\mu}(k)$, $V_{\mu}(k)$ and $V_{\mu}(k)(\varepsilon)$ as $(\cG_{\infty},K_{\infty})$-modules by means of the projection $G(F_\infty)\ra G(F_\sigma)$.
The exact sequences \eqref{exseqdisc1} and \eqref{exseqdisc2} provide the connection morphisms
\[
\partial_\sigma^{\epsilon_\sigma}:H^i(G(F),\cA(V\otimes D_{\mu}(k),\C))\longrightarrow H^{i+1}(G(F),\cA(V\otimes V_{\mu}(k)(\epsilon_\sigma),\C)),
\]
for any of the two characters $\epsilon_\sigma:G(F_\sigma)/G(F_\sigma)^+\ra\pm 1$.



Let $V_\R$ be a $(\cG_\infty,K_\infty)$-representation over $\R$ such that $V=V_\R\otimes\C$. This implies that we have a well defined complex conjugation on $V$ by conjugating on the second factor. Thus, we have a complex conjugation on $\cA(V,\C)$, given by
\[
\cA(V,\C)\ni\phi\longmapsto\bar\phi\in\cA(V,\C);\qquad\bar\phi(v)=\overline{\phi(\bar v)}.
\]

\begin{lemma}
Assume that $V=V_\R\otimes_\R\C$ for some $(\cG_\infty,K_\infty)$-module $V_\R$ over $\R$, and let $\mu\in\R$. Then, for any $\phi\in H^i(G(F),\cA(V\otimes D_{\mu}(k),\C))$, we have
\[
\overline{\partial_\sigma^{\epsilon_\sigma}(\phi)}=\epsilon_\sigma(c)\cdot\partial_\sigma^{\epsilon_\sigma}(\overline{\phi})
\]
for any $c\in G(F_\sigma)\setminus G(F_\sigma)^+$
\end{lemma}
\begin{proof}
We denote by $\cA(V\otimes D_{\mu}(k),\C)^{\pm1}\subset\cA(V\otimes D_{\mu}(k),\C)$ the subspaces where complex conjugation acts by $\pm 1$, respectively.
Since exact sequences \eqref{exseqdisc1} and \eqref{exseqdisc2} descend to exact sequences \eqref{exseqdiscreal1} and \eqref{exseqdiscreal2}, we obtain
\[
0\longrightarrow\cA(V\otimes V_{\mu}(k)(\varepsilon_\sigma),\C)^{\pm\varepsilon_\sigma}\longrightarrow\cA(V\otimes I_{\mu}(k)^{\varepsilon_\sigma},\C)^{\pm\varepsilon_\sigma}\longrightarrow\cA(V\otimes D_{\mu}(k),\C)^{\pm1}\longrightarrow 0.
\]
Hence the connection morphism satisfies
\[
\delta_\sigma^{\epsilon_\sigma}\left(H^i(G(F),\cA(V\otimes D_{\mu}(k),\C)^{\pm1})\right)\subseteq H^{i+1}(G(F),\cA(V\otimes V_{\mu}(k)(\varepsilon_\sigma),\C)^{\pm\varepsilon_\sigma})
\]
and the result follows.
\end{proof}


\subsection{Explicit computation of the connection morphisms}

Let us consider the section $s:V_{\mu}(k)_\C\rightarrow I_{\mu}(k)$ of \eqref{sect}.

We compute that
\begin{eqnarray*}
\langle f_m, \iota P_n\rangle_I&=&\int_{S^1}e^{mi\theta}P_n(-\sin\theta,\cos\theta)d\theta\\
&=&\int_{S^1}i^{2n-k+2}e^{mi\theta}e^{ni\theta}e^{(n-k+2)i\theta}d\theta=\pi i^{2n-k+2}\delta(2n-k+2+m),
\end{eqnarray*}
where $\delta(n)$ is the Dirac delta.
Thus, $g_\infty^{-1}P=\sum_{n=0}^{k-2}\alpha_n(g_\infty)P_n$, where
\[
\alpha_n(g_\infty)=\frac{i^{k-2-2n}}{\pi}\langle f_{k-2-2n},\iota(g_\infty^{-1}P)\rangle_I=\frac{i^{k-2-2n}}{\pi}\langle g_\infty f_{k-2-2n},\iota(P)\rangle_I.
\]
Since $V_{\mu}(k)$ is generated by $\{P_0,\cdots,P_{k-2}\}$, we deduce that $\alpha_n=0$ unless $n\in\{0,\cdots,k-2\}$.
Since matrix coefficient morphism are $(\cG_\sigma,K_\sigma)$-module morphisms by \S\ref{matcoef}, we can compute on the one side
\begin{eqnarray*}
R\alpha_n(g_\infty)&=&\frac{i^{k-2-2n}}{\pi}\langle g_\infty (Rf_{k-2-2n}),\iota(P)\rangle_I\\
&=&(k-n-1)\frac{i^{k-2-2n}}{\pi}\langle g_\infty f_{k-2n},\iota(P)\rangle_I=(n+1-k)\alpha_{n-1}(g_\infty),\\
L\alpha_n(g_\infty)&=&\frac{i^{k-2-2n}}{\pi}\langle g_\infty (Lf_{k-2-2n}),\iota(P)\rangle_I\\
&=&(n+1)\frac{i^{k-2-2n}}{\pi}\langle g_\infty f_{k-2n-4},\iota(P)\rangle_I=-(n+1)\alpha_{n+1}(g_\infty).
\end{eqnarray*}
On the other side, we have that $s(P_n)=\frac{2^{k-2}}{\pi}\binom{k-2}{n}^{-1}f_{2n-k+2}$. Hence,
\begin{eqnarray*}
(n<k-2)\;Rs(P_n)&=&\frac{2^{k-2}}{\pi}\binom{k-2}{n}^{-1}(n+1)f_{2n-k+4}=(k-2-n)s(P_{n+1}),\\
(n>0)\;\;Ls(P_n)&=&\frac{2^{k-2}}{\pi}\binom{k-2}{n}^{-1}(k-n-1)f_{2n-k}=ns(P_{n-1}).
\end{eqnarray*}

Assume that $\tilde\phi\in \cA(V\otimes I_{\mu}(k),\C)$ and the action of $(\cG_\sigma,K_\sigma)$ on $V$ is trivial.
For any $f\in I_{\mu}(k)$ and $v\in V$, we will usually denote by $\tilde\phi_v(f)$ the expression $\tilde\phi(v\otimes f)$.
We aim to compute
\[
h_P(g_\infty):=\tilde\phi_v(s(g_\infty^{-1}P))(g_\infty,g),\qquad g_\infty\in G(F_\sigma)^+\simeq \GL_2(\R)^+,
\]
for all $g\in G(\A^\sigma)$, $P\in V_{\mu}(k)$, and $v\in V$. Since $h_P(g_\infty)=\sum_{n=0}^{k-2}\alpha_n(g_\infty)\tilde\phi_v(s(P_n))(g_\infty,g)$, we compute
\begin{eqnarray*}
R h_P&=&\sum_{n=0}^{k-2}\left((R\alpha_n)\tilde\phi_v(s(P_n))+\alpha_n\tilde\phi_v(Rs(P_n))\right)=\\
&=&\sum_{n=0}^{k-2}(n+1-k)\alpha_{n-1}\tilde\phi_v(s(P_n))+\frac{k-1}{2^{2-k}\pi}\alpha_{k-2}\tilde\phi_v(f_{k})+\sum_{n=0}^{k-3}(k-2-n)\alpha_n\tilde\phi_v(s(P_{n+1}))\\
&=&\frac{k-1}{\pi}2^{k-2}\left(\alpha_{k-2}\tilde\phi_v(f_{k})-\alpha_{-1}\tilde\phi_v(f_{2-k})\right)=\frac{k-1}{\pi}2^{k-2}\alpha_{k-2}\tilde\phi_v(f_{k}),\\
L h_P&=&\sum_{n=0}^{k-2}\left((L\alpha_n)\tilde\phi_v(s(P_n))+\alpha_n\tilde\phi_v(Ls(P_n))\right)=\\
&=&\sum_{n=0}^{k-2}(-n-1)\alpha_{n+1}\tilde\phi_v(s(P_n))+\frac{k-1}{2^{2-k}\pi}\alpha_{0}\tilde\phi_v(f_{-k})+\sum_{n=1}^{k-2}n\alpha_n\tilde\phi_v(s(P_{n-1}))\\
&=&\frac{k-1}{\pi}2^{k-2}\left(\alpha_{0}\tilde\phi_v(f_{-k})-\alpha_{k-1}\tilde\phi_v(f_{k-2})\right)=\frac{k-1}{\pi}2^{k-2}\alpha_{0}\tilde\phi_v(f_{-k}),
\end{eqnarray*}

Notice that $R=e^{2i\theta}2iy(\frac{\partial}{\partial\tau}-\frac{1}{4y}\frac{\partial}{\partial\theta})$ and $L=-e^{-2i\theta}2iy(\frac{\partial}{\partial\bar\tau}-\frac{1}{4y}\frac{\partial}{\partial\theta})$ with $\tau=x+iy$.
Since $s$ is a morphism of $\SO(2)\R^+$-modules, $h_P$ is a function of $\GL_2(\R)^+/\SO(2)\R^+\simeq\cH$, thus $h_P$ is a function on $\tau$ and $\bar\tau$.  Let us compute $\frac{\partial h_P}{\partial\tau}$ and $\frac{\partial h_P}{\partial\bar\tau}$: By \eqref{prodVprodI} and \eqref{P-1tau},
\begin{eqnarray*}
\frac{\partial h_P}{\partial\tau}(\tau,\bar\tau)&=&\frac{y^{-1}e^{-4\pi i\theta}}{2i}R(h_P)=\frac{(k-1)}{2\pi iye^{2i\theta}}i^{2-k}\langle g_\infty P_0,P\rangle\tilde\phi_v(f_{k})\\
&=&\frac{(k-1)}{2\pi iye^{2i\theta}}(g_\infty^{-1}P)(1,-i)\tilde\phi_v(f_{k})=\frac{(k-1)}{2\pi i}P(1,-\tau)\frac{\tilde\phi_v(f_{k})}{f_k}(\tau,\bar\tau,g),\\
\frac{\partial h_P}{\partial\bar\tau}(\tau,\bar\tau)&=&\frac{-y^{-1}e^{2i\theta}}{2i}L(h_P)=\frac{(1-k)}{2\pi iye^{-2i\theta}}i^{k-2}\langle g_\infty\omega P_{0},P\rangle\tilde\phi_v(f_{-k})\\
&=&\frac{(1-k)}{2\pi iye^{-2i\theta}}( g_\infty^{-1}P)(1,i)\tilde\phi_v(f_{-k})=\frac{(1-k)}{2\pi i}P(1,-\bar\tau)\frac{\tilde\phi_v(f_{-k})}{f_{-k}}(\tau,\bar\tau,g),
\end{eqnarray*}
by Remark \ref{no-G-pairing},
where $\tilde\phi_v(f_{k})f_k^{-1}$ and $\tilde\phi_v(f_{-k})f_{-k}^{-1}$ are seen as functions of $\GL_2(\R)^+/\SO(2)\R^+\simeq\cH$. A similar (and classical) calculation shows that $\tilde\phi_v(f_{k})f_k^{-1}$ and $\tilde\phi_v(f_{-k})f_{-k}^{-1}$ are holomorphic and anti-holomorphic, respectively.

For any $P\in V_{\mu}(k)$, $g\in G(\A^\sigma)$, $v\in V$ and $\phi\in \cA(V\otimes D_{\mu}(k),\C)$ the expressions
\begin{eqnarray}
\omega_\phi(P,v,g)(\tau)&:=&P(1,-\tau)\frac{\phi(v\otimes f_{k})}{2\pi i f_{k}}(\tau,g)d\tau,\label{difhol}\\
\bar\omega_\phi(P,v,g)(\bar\tau)&:=&P(1,-\bar\tau)\frac{\phi(v\otimes f_{-k})}{2\pi i f_{-k}}(\bar\tau,g)d\bar\tau\label{difantihol},
\end{eqnarray}
define holomorphic and anti-holomorphic forms in $\cH$, respectively. Moreover, it is easy to check that
\[
\omega_\phi(P,v,g)(\gamma^{-1}\tau)=\omega_{\gamma\phi}(\gamma P, v,\gamma g)(\tau),\qquad \bar\omega_\phi(P,v,g)(\gamma^{-1}\bar\tau)=\bar\omega_{\gamma\phi}(\gamma P, v,\gamma g)(\bar\tau),
\]
for any $\gamma\in G(F)\cap G(F_\sigma)^+$.
Assume that $c_\phi\in H^{n}(G(F),\cA(V\otimes D_{\mu}(k),\C))$ is represented by the $n$-cocycle $\phi:G(F)^{n}\rightarrow \cA(V\otimes D_{\mu}(k),\C)$.
Then $\partial_\sigma^{\epsilon_\sigma}(c_\phi)$ is represented by the $(n+1)$-cocycle $d^n\tilde\phi$, where $\tilde\phi(\underline\gamma)\in\cA(V\otimes I_{\mu}(k)_\C,\C)$ is any preimage of $\phi(\underline\gamma)$ for all $\underline\gamma\in G(F)^n$. We consider the $(n+1)$-cocycle $\partial_\sigma^{\varepsilon_\sigma}(\phi)=d^n\tilde\phi-d^n b$, where $b(\underline\gamma)(g)(v\otimes P)=\tilde\phi(\underline\gamma)(v\otimes s(P))(1,g)$. We compute, for all $P\in V_\mu(k)$, $v\in V$, $\underline\gamma=(\gamma_1,\cdots,\gamma_n)\in G(F)^n$, $\alpha\in G(F)$ and $g\in G(\A^\sigma)$,
\begin{eqnarray*}
\partial_\sigma^{\epsilon_\sigma}\phi(\alpha,\underline\gamma)(g)(v\otimes P)&=&\alpha\left((\tilde\phi-b)(\underline\gamma)\right)(v\otimes s(P))(1,g)+\\
&&+\sum_{i=1}^{n+1}(-1)^i(\tilde\phi-b)(\alpha\underline\gamma_i)(v\otimes s(P))(1,g),
\end{eqnarray*}
where $\alpha\underline\gamma_i=(\alpha,\gamma_1,\cdots,\gamma_{i-1}\gamma_{i},\cdots,\gamma_{n})$ for $i=1,\cdots,n$, and $\alpha\underline\gamma_{n+1}=(\alpha,\gamma_1,\cdots,\gamma_{n-1})$. Since $(\tilde\phi-b)(\underline\gamma)(v\otimes s(P))=0$ by construction, we obtain
\begin{eqnarray*}
\partial_\sigma^{\epsilon_\sigma}\phi(\alpha,\underline\gamma)(g)(v\otimes P)&=&\alpha\left(\tilde\phi(\underline\gamma)\right)(v\otimes s(P))(1,g)-\alpha\left(b(\underline\gamma)\right)(v\otimes P)(g)\\
&=&\tilde\phi(\underline\gamma)(v\otimes s(P))(\alpha^{-1},\alpha^{-1}g)-\tilde\phi(\underline\gamma)(v\otimes s(\alpha^{-1}P))(1,\alpha^{-1}g)
\end{eqnarray*}
Since $\tilde\phi(\underline\gamma)(v\otimes f_{k})=\phi(\underline\gamma)(v\otimes f_k)$ and $\tilde\phi(\underline\gamma)(v\otimes f_{-k})=\epsilon_\sigma(c)\phi(\underline\gamma)(v\otimes f_{-k})$, we deduce from the above computations that, for any $\alpha\in G(F)\cap G(F_\sigma)^+$,
\begin{eqnarray*}
\partial_\sigma^{\epsilon_\sigma}\phi(\alpha,\underline\gamma)(g)(v\otimes P)&=&(k-1)\int_i^{\alpha^{-1}i}\omega_{\phi(\underline\gamma)}(\alpha^{-1}P,v,\alpha^{-1}g)-\epsilon_\sigma(c) \bar\omega_{\phi(\underline\gamma)}(\alpha^{-1}P,v,\alpha^{-1}g)\\
&=&(1-k)\int_i^{\alpha i}\omega_{\alpha\phi(\underline\gamma)}(P,v,g)-\epsilon_\sigma(c) \bar\omega_{\alpha\phi(\underline\gamma)}(P,v,g).
\end{eqnarray*}

\begin{remark}
We have a well defined action of $G(F)/G(F)^+$ on $H^r(G(F)^+,M)$, for any $G(F)$-module $M$, given by
\[
c^\gamma(\alpha_1,\cdots,\alpha_r)=\gamma\left(c(\gamma^{-1}\alpha_1\gamma,\cdots,\gamma^{-1}\alpha_r\gamma)\right),
\]
for $\gamma\in G(F)$, and $\alpha_i\in G(F)^+$. The image of the restriction map
\[
H^r(G(F),M)\longrightarrow H^r(G(F)^+,M)
\]
lies in $H^0(G(F)/G(F)^+,H^r(G(F)^+,M))$.
\end{remark}

Let $\psi:\A^\times/F^\times\rightarrow\C^\times$ be a Hecke character such that, for any archimedean place $\sigma_i:F\hookrightarrow\R$, $\psi_{\sigma_i}(x)={\rm sign}(x)^{k_i}|x|^{\mu_i}$. 
Let $D_{\psi}(\underline{k})$ be the $(\cG_\infty,K_\infty)$-module obtained by making the tensor product of $D_{\mu_{i}}(k_i)$ at the place $\sigma_i$, if $\sigma_i\in\Sigma$, and $V_{\mu_{j}}(k_j)$ at the place $\sigma_j$, if $\sigma_j\not\in \Sigma$. An element of $D_{\psi}(\underline{k})$ is $f_{\underline{k}}\otimes P^{\Sigma}$, where $f_{\underline{k}}=\bigotimes_{\sigma_i\in\Sigma} f_{k_i}$, $f_{k_i}\in D_{\mu_{i}}(k_i)$ are the elements defined above, and $P^\Sigma\in\bigotimes_{\sigma_i\not\in\Sigma}V_{\mu_{i}}(k_i)$. Let $V_{\psi}(\underline{k})$ be the $(\cG_\infty,K_\infty)$-module obtained by making the tensor product of $V_{\mu_{i}}(k_i)$ at all the places $\sigma_i$.
For any character $\epsilon: G(F_\infty)/G(F_\infty)^+\simeq G(F)/G(F)^+\rightarrow\pm 1$, we denote by $H^r(G(F)^+,\cA(V_{\psi}(\underline{k}),\C))(\epsilon)$ the $\varepsilon$-isotypical component, namely, the subspace of $H^r(G(F)^+,\cA(V_{\psi}(\underline{k}),\C))$ such that the action of $G(F)/G(F)^+$ is given by the character. By the above remark, the restriction map provides an isomorphism
\[
H^r(G(F),\cA(V_{\psi}(\underline{k})(\varepsilon),\C))\simeq H^r(G(F)^+,\cA(V_{\psi}(\underline{k}),\C))(\epsilon).
\]
Using the above computations, we aim to give an explicit formula for the connection morphism:
\begin{theorem}\label{mainres}
Let $\phi$ be a weight $\underline{k}$ automorphic form of $G(\A)$ with central character $\psi$. Then $\phi$ defines an element of $\phi\in H^0(G(F),\cA(D_{\psi}(\underline{k}),\C))$. For a choice of signs at the places at infinity
\[
\epsilon: G(F_\infty)/G(F_\infty)^+\simeq G(F)/G(F)^+\rightarrow\pm 1,
\]
the composition of the connection morphisms $\delta_{\epsilon_\sigma}$, for $\sigma\in\Sigma$,
\[
\partial_\epsilon:H^0(G(F),\cA(D_{\psi}(\underline{k}),\C))\longrightarrow H^r(G(F),\cA(V_{\psi}(\underline{k})(\epsilon),\C))\simeq H^r(G(F)^+,\cA(V_{\psi}(\underline{k}),\C))(\epsilon),
\]
can be computed as follows:
\[
\partial_\epsilon\phi=\prod_{j+1}^r(1-k_j)\sum_{\gamma\in G(F)/G(F)^+}\epsilon(\gamma)\partial\phi^\gamma,
\]
where $\partial\phi\in H^r(G(F)^+,\cA(V_{\psi}(\underline{k}),\C))$ is the class of the cocycle
\begin{equation*}\label{cocycleex}
(G(F)^+)^r\ni(g_1,g_2,\cdots, g_r)\longmapsto \int_{\tau_1}^{g_1\tau_1}\cdots\int^{g_1\cdots g_r\tau_r}_{g_1\cdots g_{r-1}\tau_r}P_{\Sigma}(1,-\underline{z})\frac{\phi( f_{\underline k}\otimes P^{\Sigma})}{ (2\pi i)^rf_{\underline{k}}}(\underline{z},1,g)d\underline{z},
\end{equation*}
for any $\underline{P}=P^{\Sigma}\otimes P_{\Sigma}\in V_{\underline\mu}(\underline{k})$, $\underline{z}=(z_1,\cdots,z_r),(\tau_1,\cdots,\tau_r)\in\cH^r$.
\end{theorem}
\begin{proof}
Let $S\subset \Sigma$ be a subset of archimedean places such that $\#S=s< r$. Assume that $\sigma=\sigma_j\in \Sigma\setminus S$ and let $k$ be its corresponding weight and $\mu=\mu_j$. Let $V=\bigotimes_{\sigma_i\in S'}D_{\mu_{i}}(k_i)\otimes \bigotimes_{\sigma_i\in \infty\setminus (\Sigma\setminus S)}V_{\mu_{i}}(k_i)(\varepsilon_{\sigma_i})$, where $S'=\Sigma\setminus (S\cup\{\sigma\})$. Thus, the composition of the connection morphisms corresponding to $\sigma_i\in S$, provides a morphism
\[
\delta_S:H^0(G(F),\cA(D_{\psi}(\underline{k}),\C))\longrightarrow H^s(G(F),\cA(V\otimes D_{\mu}(k),\C)).
\]
By the previous computations, if $P\in V_{\mu}(k)$, $g\in G(\A^\sigma)$, $v\in V$,
\[
\partial_\sigma^{\epsilon_\sigma}\partial_S\phi(\alpha,\underline\gamma)(g)(v\otimes P)=(1-k)\int_i^{\alpha i}\omega_{\alpha\partial_S\phi(\underline\gamma)}(P,v,g)-\epsilon_\sigma(c) \bar\omega_{\alpha\partial_S\phi(\underline\gamma)}(P,v,g).
\]

Notice that, letting $c\in (G(F)\cap G(F_S)^+)\setminus (G(F)\cap G(F_\sigma)^+)$,
by means of the change of variables $\tau=g_\infty i\mapsto z=c\bar\tau=c g_\infty \omega i\in\cH$, where $g_\infty\in\GL_2(\R)^+$, we obtain that
\begin{eqnarray*}
\int_i^{\alpha i} \bar\omega_{\alpha\partial_S\phi(\underline\gamma)}(P,v,g)&=&\int_i^{\alpha i} P(1,-\bar\tau)\frac{\alpha\partial_S\phi(\underline\gamma)(v\otimes f_{-k})(g_\infty,g)}{f_{-k}(g_\infty)}d\bar\tau\\
&=&-\int_{c(-i)}^{c\alpha (-i)} (c\ast P)(1,-z)\frac{c\alpha\partial_S\phi(\underline\gamma)(v\otimes f_{k})(cg_\infty\omega,cg)}{f_{k}(cg_\infty\omega)}dz\\
&=&-\int_{\tau_0}^{c\alpha c^{-1}\tau_0} \omega_{c\alpha\partial_S\phi(\underline\gamma)}(c\ast P,v,cg),
\end{eqnarray*}
where $\tau_0=c(-i)$, but in fact, this last expression does not depend on the choice of $\tau_0$ because $\cH$ is simply connected. Since $c\partial_S\phi(\underline\gamma)=\partial_S\phi(c\underline\gamma c^{-1})$ because $c\in G(F)\cap G(F_S)^+$, we obtain that
\[
\partial_\sigma^{\epsilon_\sigma}\partial_S\phi(\alpha,\underline\gamma)=(1-k)(r(\alpha,\underline\gamma)+\varepsilon(c)c^{-1}r(c\alpha c^{-1},c\underline\gamma c^{-1})),
\]
where $r$ is the cocycle
\[
r(\alpha,\underline\gamma)(g)(v\otimes P)=\int_i^{\alpha i}\omega_{\alpha\partial_S\phi(\underline\gamma)}(P,v,g).
\]
Applying a simple induction on $S$ we obtain the desired result.
\end{proof}

\bibliographystyle{plain}
\bibliography{ES-conn}

\end{document}